\documentclass[12pt]{amsart}

\usepackage{graphicx}
\usepackage{amsmath}
\usepackage{amscd}
\usepackage{amsfonts}
\usepackage{amssymb}
\usepackage{color}
\usepackage{fullpage}
\usepackage{mathtools}
\usepackage[hypertexnames=false, colorlinks, citecolor=red,linkcolor=blue, urlcolor=black]{hyperref}
\usepackage{tikz-cd}
\usepackage[capitalise]{cleveref}
\usepackage[backref=true, backend=biber,style=numeric,sorting=nty, doi = false, url=false]{biblatex}
\numberwithin{equation}{section}

\newtheorem{theorem}{Theorem}[section]
\newtheorem{lemma}[theorem]{Lemma}
\newtheorem{proposition}[theorem]{Proposition}

\newtheorem{corollary}[theorem]{Corollary}

\newtheorem{question}[theorem]{Question}

\theoremstyle{definition}
\newtheorem{example}[theorem]{Example}
\newtheorem{remark}[theorem]{Remark}

\newcommand{\be}{\begin{equation}}
\newcommand{\ee}{\end{equation}}
\newcommand{\bes}{\begin{equation*}}
\newcommand{\ees}{\end{equation*}}

\newcommand{\Rad}{\mathcal{R}}

\newcommand{\bA}{\mathbb{A}}

\newcommand{\bN}{\mathbb{N}}
\newcommand{\bQ}{\mathbb{Q}}

\newcommand{\bZ}{\mathbb{Z}}
\newcommand{\Gm}{\mathbb{G}_m}

\newcommand{\disc}{\operatorname{disc}}

\newcommand{\Ad}{\operatorname{Ad}}
\newcommand{\im}{\operatorname{im}}
\newcommand{\Gal}{\operatorname{Gal}}

\newcommand{\PGL}{\operatorname{PGL}}

\newcommand{\Br}{\operatorname{Br}}
\newcommand{\GL}[1]{\operatorname{GL}_{#1}}
\newcommand{\SL}{\operatorname{SL}}
\newcommand{\SO}{\operatorname{SO}}
\newcommand{\Spin}{\operatorname{Spin}}
\newcommand{\Sp}{\operatorname{Sp}}

\newcommand{\ind}{\operatorname{ind}}
\newcommand{\Inv}{\operatorname{Inv}}

\newcommand{\Char}{\operatorname{char}}

\begin{document}

\title{Reduction of Structure to Parabolic Subgroups}
\begin{abstract}
Let $G$ be an affine group over a field of characteristic not two. A $G$-torsor is called \emph{isotropic} if it admits reduction of structure to a proper parabolic subgroup of $G$. This definition generalizes isotropy of affine groups and involutions of central simple algebras. When does $G$ admit anisotropic torsors? Building on work of J. Tits, we answer this question for simple groups. We also give an answer for connected and semisimple $G$ under certain restrictions on its root system.

\end{abstract}

\author{Danny Ofek}
\address{
University of British Columbia\\
Vancouver\;
Canada}
\email{Dannyofe@student.ubc.ca}


\maketitle

\tableofcontents

\section{Introduction}

\subsection{Motivation and context}
A linear algebraic group over a field $k$ is called \emph{isotropic} if it contains a non-trivial split torus, or equivalently a proper parabolic subgroup \cite[Corollary 4.17]{Borel1965}. This definition arises in the problem of classification of absolutely semisimple linear algebraic groups. Let $k_s$ be the separable closure of $k$. To every semisimple group $G$ over $k$ one associates a root system $\Delta$, an action of $\Gamma = \Gal(k_s/k)$ on $\Delta$ and a maximal split torus $S\subset G$. One defines the anisotropic kernel of $G$, denoted $G_0$, as the derived subgroup of the centralizer of $S$ \cite{Tits1966}. These determine $G$ up to isogeny \cite[Theorem 2]{Tits1966}. $G$ is called \emph{quasi-split} if $G_0 = \{1\}$, or equivalently if it contains a Borel subgroup defined over $k$ \cite[Proposition 16.2.2]{Springer1998}. Every semisimple group is an inner form of a unique quasi-split group \cite[Proposition 1.10]{Merkurjev1996}. Using this J. Tits proved an existence theorem which essentially reduces the classification of semisimple $k$-groups to classifying all anisotropic forms ${}_\gamma G$ where $G$ is a semisimple quasisplit group and $[\gamma]\in H^1(k,G)$ \cite[Proposition 4]{Tits1966}. This classification is intractable in full generality, for example if we take $G = \PGL_n$ it is equivalent to classifying all division algebras of degree $n$. The simplest question one could ask about the class of all anisotropic forms in $\{{}_\gamma G\mid [\gamma]\in H^1(k,G)\}$ is whether any exist at all.

\begin{question}\label{Q1}
Which quasi-split simple groups $G$ over $k$ admit an anisotropic form ${}_\gamma G$ with $[\gamma] \in H^1(K,G)$ for some field extension $k \subset K$?
\end{question}

In \cite{Tits1990} Tits answered this question for a split and simply connected $G$. The present paper was motivated by his work and our main results are generalizations of \cite[4.4.1, Theorem 2]{Tits1990}. The following definition will be fundamental. We call a torsor in $H^1(k,G)$ \emph{isotropic} if it admits reduction of structure to some proper parabolic subgroup $P\subset G$.
When $G$ is the automorphism group of some "algebraic structure", torsors in $H^1(k,G)$ correspond to forms of that structure (see \cite[Chapter III, Section 1]{Serre1997} for details). Through this correspondence isotropy of torsors generally gains a concrete and intuitive meaning. Two instructive examples are when the algebraic structure is a central simple algebra or a quadratic form. These are worked out in \cref{criterion for isotropy An} and \cref{criterion for isotropy D5} respectively.
We say a linear algebraic group $G$ is \emph{strongly isotropic} if all its torsors over any field extension are isotropic. If $G$ is quasisplit a torsor $[\gamma]\in H^1(K,G)$ is isotropic if and only if ${}_\gamma G$ is isotropic. Therefore the following is a generalization of \cref{Q1}.
\begin{question}\label{Q2}
What are the strongly isotropic connected reductive linear algebraic groups over $k$?
\end{question}
This is the main question we will explore in this paper. It can be partially reduced to the case of semisimple groups, see \cref{reduction to semisimple} and \cref{Conditions on Radical}. 

\subsection{Main results}
In \cref{Classification of simple strongly isotropic algebraic groups} we prove the following classification theorem for simple strongly isotropic groups. 

\begin{theorem}\label{simple strongly isotropic}
Let $k$ be a field with $\Char(k)\neq 2$. The simple strongly isotropic groups over $k$ are $\Sp_{2n}, \Spin(q)$  where $q$ is a ten dimensional regular quadratic form with trivial discriminant and split Clifford algebra and $\SL_n(D)/\mu_d$ where $D$ is a central division algebra over $k$ and $n,d$ are natural numbers such that $n>1$, $d\mid n\deg(D)$ and some prime divisor of $n$ does not divide $d$.
\end{theorem}
The proof relies heavily on known classification results. We first explain how to use \cite{Tits1990} to reduce the problem to groups of type $A_n,C_n$ or $D_5$. In each of these cases there  are algebraic structures classified by $H^1(k,\overline{G})$. Using this correspondence we get a notion of isotropy for these algebraic structures coming from isotropy of torsors in $H^1(k,\overline{G})$. We then construct appropriate "anisotropic" algebraic structures over field extensions to prove the existence of anisotropic torsors.  The index reduction formulas of A. S. Merkurjev, I. A. Panin, and A. R. Wadsworth are a central component in our constructions. \par 

The problem of classifying strongly isotropic semisimple groups is more involved. In \cref{semisimple section} we will prove the following partial classification.
\begin{theorem}\label{strongly isotropic quotient}
Let $G$ be a semisimple group over a field $k$ with $\Char(k)\neq 2$. Assume for any simple factor of $G$ of type $A_{n-1}$ the integer $n$ is squarefree. Then $G$ is strongly isotropic if and only if it admits a simple strongly isotropic quotient.
\end{theorem}
Note that \cref{strongly isotropic quotient} fails if we allow simple factors of type $A_{n-1}$ when $n$ is not square-free; see \cref{example no quotient}. We classify strongly isotropic split semisimple groups of type $A$ in \cref{Strongly isotropic GLn's}.

\subsection{Notational conventions} All connected linear algebraic groups are assumed to be smooth. A morphism of algebraic groups $f:G\to H$ is surjective if the induced morphism of abstract groups $G(\overline{k})\to H(\overline{k})$ is surjective. We call $f$ a quotient morphism if it is  separable as a morphism of varieties and surjective. We denote by $\Ad:G\to \overline{G}$ the natural projection onto the adjoint group when $G$ is reductive.

For any field $k\subset K$ we denote by $H^1(K,G)= H^1(\Gal(K_s/K),G(K_s))$ the first non-abelian cohomology of $G$ over $K$. We usually denote cohomology classes by $[\gamma]$ where $\gamma$ is a cocycle in $G$. We say a torsor $[\gamma]\in H^1(K,G)$ \emph{admits reduction of structure} to a $K$-subgroup $H\subset G_K$ if $[\gamma]$ lies in the image of the natural morphism $H^1(K,H)\to H^1(K,G)$. If $H$ is defined over $k$ and $H^1(K,H)\to H^1(K,G)$ is surjective for all field extensions $k\subset K$ we say $G$ \emph{admits reduction of structure} to $H$. We denote the Brauer equivalence class of a central simple algebra over $k$ by $[A]\in \Br(k)$. Additive notation is used for Brauer groups.

\subsection{Acknowledgements}
I would like to thank my supervisor Zinovy Reichstein for introducing me to \cref{Q2} and for guiding me through this project.

\section{A theorem of Jacques Tits}
In this section we explain how to deduce the following theorem from \cite{Tits1990}. 

\begin{theorem}\label{simply connected split groups}
The simply connected split strongly isotropic simple groups over a field $k$ with $\Char(k)\neq 2$ are exactly $\SL_{n},\Sp_{2n}$ and $\Spin_{10}$.
\end{theorem}

We start by stating the connection between isotropic groups and isotropic torsors. Let $G$ be a linear algebraic group over a field $k$ with $\Char(k)\neq 2$.

\begin{lemma}\label{isotropy and flags}
Let $[\gamma]\in H^1(k,G)$ be a torsor and $P\subset G$ a parabolic subgroup. The following are equivalent:
\begin{enumerate}
    \item $[\gamma]$ admits reduction of structure to $P$.
    \item $_\gamma G$ has a parabolic $P'$ of the same type as $P$ (i.e. conjugate to it over $\overline{k}$) which is defined over $k$.
    \item The variety $ {}_\gamma (G / P)$ has a $k$-rational point.
\end{enumerate}
In particular if $[\gamma]$ is isotropic then ${}_\gamma G$ is isotropic. When $G$ is quasisplit the converse implication holds as well.
\end{lemma}
\begin{proof}
The implications $(1)\iff (2)$ follows from \cite[III.2.2, Lemma 1]{Serre1997} and the fact that $P$ is self normalizing. The implication $(1)\iff (3)$ follows from \cite[Chapter I, Proposition 37]{Serre1997}.
\end{proof}

Assume $G$ is simply connected, split and simple. The groups ${}_\gamma G$, where $[\gamma]\in H^1(k,G)$ are what Tits defined as \emph{strongly inner forms} of $G$. \cref{isotropy and flags} shows $G$ is strongly isotropic if and only if every strongly inner form of $G$ is isotropic over any field extension. Therefore \cref{simply connected split groups} is equivalent to the following proposition.

\begin{proposition}\label{anisotropic inner forms}
A split simple simply connected group $G$ admits an anisotropic strongly inner form over a field extension if and only if it is not of type $A_n$, $C_n$ or $D_5$.
\end{proposition}

\begin{proof}
The case, where the base field $k = \bQ$ is the field of rational numbers is proved in \cite{Tits1990}; see \cite[4.4.1, Theorem 2]{Tits1990}.  The same arguments work for arbitrary $k$ with only minor adjustments. For the sake of completeness, we spell out these minor adjustments below. \par If $G$ is of type $A_n$, $C_n$ or $D_5$ any strongly inner form of $G$ is isotropic by \cite[4.4.1, Theorem 2]{Tits1990}. We assume this is not the case and prove $G$ admits an anisotropic strongly inner form over a field extension of $k$. We split into cases according to the type of $G$.
\begin{itemize}
    \item[Type $E_6, E_7$:] By \cite[Propositions 2,3]{Tits1990} it suffices to construct a field extension $k\subset K$ such that $K$ admits central division algebras of index $9$ and central division algebras of index $16$ and period $4$. This follows for example from \cref{existence of division}.
    \item[Type $E_8,G_2, F_4$:] This is explained in the beginning of \cite[Section 5]{Tits1990}.
    \item[Type $D_n, B_n$:] By \cite[4.4.2]{Tits1990} it suffices to prove the existence of a $d$-dimensional anisotropic quadratic form $q$  over a field extension $K$ with trivial Witt invariant and discriminant for all $7\leq d\neq 10$ (see the paragraph preceding \cref{properties of Cliff invar} for the definition of the Witt invariant). This was done by Tits in \cite[Proposition 6]{Tits1990}.
\end{itemize}
\end{proof}

\section{Preliminaries}\label{preliminaries}
${}$ \\ In this section we make a few elementary observations and prove some technical lemmas. We will use the next elementary fact repeatedly.

\begin{lemma}\label{useful lemma}
Let $N\subset H \subset G$ be algebraic groups over $k$ with $N$ normal in $G$. A torsor $[\gamma]\in H^1(k,G)$ admits reduction of structure to $H^1(k,H)$ if and only if $[\gamma N]\in H^1(k,G/N)$ admits reduction of structure to $H^1(k,H/N)$. 
\end{lemma}
\begin{proof}
If $[\gamma]$ admits reduction of structure to $H^1(k,H)$ then  $[\gamma N]\in H^1(k,G/N)$ admits reduction of structure to $H^1(k,H/N)$ by commutativity of the square:
$$
\begin{tikzcd}
{H^1(k,G)} \arrow[r]            & {H^1(k,G/N)}     \\
{H^1(k,H)} \arrow[r] \arrow[u] & {H^1(k,H/N)} \arrow[u] 
\end{tikzcd}$$
Conversely, if $\gamma N$ is homologous to a cocycle with image in $H/N$ there exists $g\in G$ such that for all $\sigma\in\Gal(k)$: $$g\gamma(\sigma)\sigma(g^{-1})N\subset H N= H.$$
This implies that $g\gamma(\sigma)\sigma(g^{-1})\in H$ for all $\sigma$ and so $[\gamma]\in H^1(k,H)$. 
\end{proof} 

We draw two direct corollaries relevant to our interests. 

\begin{corollary}\label{covering thm1}
Let $\pi: G\to H$ be a surjective morphism of algebraic groups. Let $P\subset G$ be a parabolic subgroup with $\ker\pi\subset P$ and $[\gamma]\in H^1(k,G)$ some torsor.
\begin{enumerate}
    \item $[\gamma]$ admits reduction of structure to $P$ if and only if $\pi_*[\gamma]$ admits reduction of structure to $\pi(P)$.
    \item Assume that either $\pi$ is a quotient map or $\ker\pi$ is central and $G$ is reductive. If $H$ is strongly isotropic so is $G$.
\end{enumerate}
\end{corollary}
\begin{proof}
$(1)$ is a particular case of \cref{useful lemma}. $(2)$ follows from $(1)$ since $\pi^{-1}(P)$ is a proper parabolic of $G$ for any proper parabolic $P$ of $H$ by \cite[Theorem V.22.6]{Borel1991}, \cite[Proposition II.6.13]{Borel1991} and \cite[AG 17.3]{Borel1991}.
\end{proof}
The following example shows the restrictions on $\pi$ in the corollary are necessary.
\begin{example}
Let $G=\GL{1}(D)$ for some central division algebra $D$ of period $p$ over a field of characteristic $p>0$. Denote by $\pi:G\to G^{(p)}$ the  relative Frobenius homomorphism (see \cite[2.28]{Milne2017}). It is always surjective according to our definition \cite[Proposition 2.29]{Milne2017}.  It is easy to see $G^{(p)}$ is isomorphic to $\GL{1}(D^{(p)})$ where $D^{(p)}$ is the Frobenius twist of $D$ (see \cite[4.1.1]{Jacobson1996} for a definition). By \cite[Theorem 4.1.2]{Jacobson1996} $D^{(p)}$ is split and so $G^{(p)}$ is strongly isotropic by Hilbert 90. Meanwhile $G$ is anisotropic and in particular it is not strongly isotropic.
\end{example}

Since  the  center  of $G$ is  contained  in  any  parabolic subgroup $P\subset G$ the  previous  corollary  implies  the  following  one  which  will  be  used throughout this paper (sometimes implicitly).
 \begin{corollary}\label{adjoint thm1}
Assume $G$ is reductive. Let $\Ad:G\to \overline{G}$ be the canonical surjection. A torsor $[\gamma]\in H^1(k,G)$ is isotropic if and only if $\Ad_*[\gamma]$ is isotropic.
\end{corollary}

By definition any torsor of a strongly isotropic group $G$ admits reduction of structure to some proper parabolic. In fact one can find a "universal" parabolic $P\subsetneq G$ such that $H^1(K,P)\to H^1(K,G)$ is surjective for any field extension $K$. In such a situation we say $G$ admits reduction of structure to $P$.
\begin{lemma}\label{versal torsors}
If $\Char(k)>0$ assume $G$ is reductive. $G$ is strongly isotropic if and only if it admits reduction of structure to some proper parabolic subgroup $P\subset G$.
\end{lemma}
\begin{proof}
Let $U$ be a versal torsor of $G$ which is the fiber at the generic point of a $G$-torsor over a smooth irreducible base $Q\to X$, see \cite[Definition I.5.1]{GaribaldiS.MerkurjevA.Serre2003} for a definition and \cite[I.5.3]{GaribaldiS.MerkurjevA.Serre2003} for existence. We can associate to $U$ in a natural way a cohomology class $u\in H^1(K,G)$ where $K= k(X)$ \cite[I.5.2]{Serre1997}. If $G$ is strongly isotropic then $u$ admits reduction of structure to some proper parabolic $P$. Therefore the twist of $(G/P)_K$ by $U$, denoted ${}_U(G/P)_K$, has a $K$-rational point (see \cite[I.5.3]{Serre1997}). Therefore there exists a $G$-equivariant map $U \to G/P_K$  \cite[Proposition 3.2]{Duncan2015}. Since $U$ is the generic fiber of $Q\to X$ this defines a $G$-equivariant dominant rational map $Q\dasharrow G/P$. Therefore in the notation of \cite{Duncan2015} $G/P$ is a very versal $G$-torsor and $H^1(L,P)\to H^1(L,G)$ is surjective for any infinite field $L$ by \cite[Proposition 7.1]{Duncan2015}. For any finite field $L$, $H^1(L,G)$ vanishes since $G$ is connected by \cite[Theorem 1, Chapter 3, Section 2]{Serre1997}. Therefore $G$ admits reduction of structure to $P$.
\end{proof}
\begin{remark}
For any reductive and connected group $G$, any parabolic subgroup $P\subset G$ and any field extension $k\subset K$ the induced morphism $H^1(K,P)\to H^1(K,G)$ is injective \cite[Section III.2.1, Exericse 1]{Serre1997}. Therefore if $G$ admits reduction of structure to $P$, then $H^1(K,P)\to H^1(K,G)$ is a bijection for all $K$.
\end{remark}

\begin{proposition}\label{reduction to semisimple}
Let $G$ be a connected algebraic group over $k$. Denote by $\Rad_u(G), \Rad(G), G'$ the $k$-unipotent radical, $k$-radical and derived subgroup of $G$ respectively.
\begin{enumerate}
    \item Assume $k$ is perfect. Then $G$ is strongly isotropic if and only if $G/\Rad_u(G)$ is strongly isotropic.
    \item Assume $G$ is reductive. If $G$ is strongly isotropic then so is $G'$. If $\Rad(G)$ is split the converse implication holds as well.
\end{enumerate}
\end{proposition}

\begin{proof}
\begin{enumerate}
    \item Assume $k$ is perfect. Then $\Rad_u(G)$ is split. Let $k\subset K$ be a field extension. Since $\Rad_u(G)_K$ is split the induced morphism $H^1(K,G)\to H^1(K,G/\Rad_u(G))$ is bijective (see \cite[Lemme 1.13]{Sansuc1981}). Therefore \cref{covering thm1} implies $G$ is strongly isotropic if and only if $G/\Rad_u(G)$ is strongly isotropic because $\Rad_u(G)$ is contained in all of $G$'s parabolics. 
    
    \item Assume $G$ is reductive. For any parabolic subgroup $P\subset G$ consider the commutative diagram:
$$
\begin{tikzcd}
{H^1(K,G')} \arrow[r, "\iota"]             & {H^1(K,G)} \arrow[r, "\pi"]                & {H^1(K,G/G')}                                  \\
{H^1(K,P\cap G')} \arrow[r] \arrow[u, "a"] & {H^1(K,P)} \arrow[u, "b"] 
\end{tikzcd}$$
If $\Rad(G)$ is split then $G/G' \cong \Rad(G)/(\Rad(G)\cap Z(G'))$ is a split torus since it is a homomorphic image of $\Rad(G)$  \cite[Theorem V.15.4]{Borel1991}. Therefore $\iota$ is surjective by Hilbert 90. Thus it suffices to show for any $[\gamma]\in H^1(K,G')$: $$[\gamma]\in \im(a) \iff \iota[\gamma]\in \im(b).$$  If  $[\gamma]\in \im(a)$ then $\iota[\gamma]\in \im(b)$ by commutativity of the diagram. Conversely, if  $\iota[\gamma]\in \im(b)$, then by \cref{isotropy and flags} ${}_\gamma (G/P)$ has a $K$-rational point. Since $PG'= G$, ${}_\gamma (G/P) $ is isomorphic over $K$ to ${}_\gamma (G'/P\cap G')$ (see \cite[Theorem 8]{Rosenlicht1956}). By \cref{isotropy and flags} it follows that $[\gamma]$ admits reduction of structure to $P\cap G'$.  This finishes the proof.

\end{enumerate}
\end{proof}
\begin{remark}\label{Conditions on Radical}
The assumption that $\Rad(G)$ is split is far from necessary for the converse implication in $(2)$ to hold. Let $G$, $G'$ be as in $(2)$ and assume $G'$ is strongly isotropic. Denote by $\partial_{G'} : H^1(K,\overline{G}) \to H^2(K,Z(G'))$ and $\partial_{Z}:H^1(K,Z(G)/Z(G')) \to H^2(K,Z(G'))$ the connecting maps induced from the exact sequences $1\to Z(G')\to G'\to \overline{G} \to 1$ and $1\to Z(G')\to Z(G) \to Z(G)/Z(G') \to 1$ respectively. It is easy to check that we have:
$$(\Ad_G)_*H^1(K,G) = \partial_{G'}^{-1}\im(\partial_Z).$$
Therefore \cref{adjoint thm1} implies $G$ is strongly isotropic if and only if for all $z\in H^1(K,Z(G)/Z(G'))$ the fiber $\partial^{-1}_{G'}(\partial_Z(z))$ consists of isotropic torsors. This holds for example whenever $\partial_Z = 0$ because $G'$ is strongly isotropic. When $G'$ is simple, one can deduce from \cref{simple strongly isotropic} and \cite[Chapter I, Proposition 44]{Serre1997} exactly what are the cohomology classes $x\in H^2(k,Z(G'))$ for which the fiber $\partial_{G'}^{-1}(x_K)$ consists of isotropic torsors for all extensions $k\subset K$. In many cases this reduces the problem of deciding whether $G$ is strongly isotropic to calculating the image of $\partial_Z$.
\end{remark}

We finish this section with an elementary lemma that will be used in \cref{semisimple section}.
\begin{lemma}\label{products and strong isotropy}
Let $G_1,G_2$ be connected algebraic groups over $k$ and put $G = G_1 \times G_2$.
\begin{enumerate}
    \item The maximal proper parabolics of $G$ are of the form $P_1 \times G_2$ or $G_1 \times P_2$ where $P_i$ is a maximal proper parabolic subgroup of $G_i$.
    \item $G$ is strongly isotropic if and only if $G_1$ or $G_2$ is strongly isotropic.
\end{enumerate}
\end{lemma}
\begin{proof}
Part $(2)$ follows from $(1)$, the fact Galois cohomology preserves products and \cref{versal torsors}. For part $(1)$ notice that $P_1\times G_2$ and $G_1\times P_2$ are clearly parabolic subgroups of $G$. Let $P\subset G$ be a proper parabolic subgroup and denote by $H_i$ the projection of $P$ onto the $i$-th coordinate. By \cite[Proposition IV.11.14]{Borel1991} and \cite[Corollary I.1.4]{Borel1991},  $H_1,H_2$ are parabolic subgroups. Therefore $P$ is contained in a parabolic of the required form.
\end{proof}

\section{Generic anisotropic constructions} Our proof \cref{simple strongly isotropic} will rely on lemmas asserting the existence of various anisotropic algebraic objects. In this section we give generic constructions of anisotropic bilinear forms and involutions on central simple algebras. For the definition of isotropy of involutions of central simple algebras see \cite[Page 72]{Knus1998}. The following elementary lemma will be used in our constructions repeatedly. 
\begin{lemma}\label{Existence of anisotropic}
Let $k$ be a field with $\Char(k) \neq 2$ and let $n$ be a natural number. There exists a quadratic form $g$ of dimension $n$ and unsigned discriminant $1$ over a finitely generated transcendental field extension $K$ such that $f_K\otimes_K g$ is anisotropic for any anisotropic quadratic form $f$ over $k$.
\end{lemma}
\begin{proof}
We can take $k\subset K$ to be a field extension generated by algebraically independent elements $x_1 ,\dots ,x_{n-1}$ and $g$ to be the quadratic form $\langle -x_1 ,\dots ,-x_{n-1}, (-1)^{n-1} \prod_{i=1,\dots, n-1}x_i\rangle$. To prove $f_K\otimes g$ is anisotropic it suffices to prove $f_K\otimes \langle\langle x_1 ,\dots ,x_{n-1}\rangle\rangle$ is anisotropic since the former is a subform of the latter. This follows by induction from \cite[Exercise 1, Chapter IX]{Lam2004}.
\end{proof}

The following corollary will be used to construct anisotropic symplectic involutions (see \cite[Definition 2.5]{Knus1998} for the definition of a symplectic involution).

\begin{corollary}\label{Quaternion hermitian}
Let $Q$ be a quaternion division algebra over $k$ and let  $n\in \bN$. There exists an anisotropic hermitian form on $Q^n_K = Q_K \times \dots \times Q_K$ for some finitely generated transcendental field extension $k\subset K$.
\end{corollary}
\begin{proof}
Denote by $n(Q)$ the norm form of $Q$. By \cref{Existence of anisotropic} there exists a transcendental field extension $k\subset K$ and a quadratic form $g = \langle a_1,\dots ,a_n\rangle$ such that $g\otimes n(Q_K)$ is anisotropic. Define a hermitian form $b$ on $Q_K^n$ by setting for any $q = (q_1,\dots q_n), q' = (q'_1,\dots ,q'_n)\in Q^n_K$:
$$b(q,q') = \sum_r a_r q_r \overline{q'_r}.$$
Clearly $b$ is hermitian. Notice that $b(q,q) = 0$ implies 
$$g\otimes n(Q_K) (q)=\sum_r a_r n(Q_K)(q_r) =0. $$
Which implies $q_r=0 $ for all $r$ since $g\otimes n(Q_K)$ is anisotropic.
\end{proof}

The next lemma will be used in \cref{Classification of simple strongly isotropic algebraic groups} to prove anisotropic $G$-torsors exist for groups $G$ of type $C_n$ which are not isomorphic to $\Sp_{2n}$.

\begin{lemma}\label{anisotropic symplectic involution}
Let $A$ be a central simple algebra of exponent two over $k$. If $A$ is not split then there exists a field extension $k\subset K$ such that $A_K$ admits an anisotropic symplectic involution.
\end{lemma}
\begin{proof}
Let $X$ be the generalized Severi-Brauer variety of right ideals of $A$ of dimension $2\deg(A)$. By \cite[Theorem 3]{Blanchet1991} the index of $A_{k(X)}$ is two. Therefore by Wedderburn's theorem $A_{k(X)}$ is isomorphic to $M_n(Q)$ for some quaternion division algebra $Q$ with $Z(Q)= k(X)$. By \cref{Quaternion hermitian} after possibly extending scalars to a transcendental extension we may assume $Q^n$ admits an anisotropic hermitian form. Then we are done by the correspondence between hermitian forms on $Q^n$ and symplectic involutions on $M_n(Q)$ (see \cite[Proposition 2.21, Theorem 4.2]{Knus1998}).
\end{proof}

Next we construct anisotropic unitary involutions (see \cite[I.2.B]{Knus1998} for the definition of a unitary involution). The idea for the reduction to split case in the proof is reproduced from a Mathoverflow post by M. Borovoi \cite{Borovoi2010}.

\begin{lemma}\label{Anisotropic hermitian involution}
Let $A$ be a simple $k$-algebra of degree $n$ such that $Z(A) = K$ where $K = k(\sqrt{\alpha})$ is a quadratic field extension of $k$. There exists a field extension $k\subset L$ and a $d$-dimensional quadratic form $q$ over $L$ of unsigned discriminant $1$ such that the following holds:
\begin{enumerate}
    \item $A\otimes_k L$ and  $M_d(K_L)$ are isomorphic as $L$-algebras.
    \item $\alpha$ is not a square in $L$.
    \item The unitary involution $\gamma \otimes_L \sigma_q$ on $M_n(K_L)$ is anisotropic, where $\sigma_q$ is any involution adjoint to $q$ and $\gamma$ is the unique non-trivial element of $\Gal(K_L/L)$. 
\end{enumerate}
\end{lemma}

\begin{proof}
By replacing $k$ with a transcendental extension we may assume $|k|=\infty$ without loss of generality. We start by constructing $L$. Let $S$ be the $k$-variety of hermitian elements in $A$ with separable reduced characteristic polynomial (defined over $K$).  There exists a canonical isomorphism: $$\varphi: (A,\sigma)\otimes_k \overline{K} \to (M_n(\overline{K})\times M_n(\overline{K})^{op}, \varepsilon),\ \  \varphi(a\otimes \beta) =(a\beta, (\sigma(a)\beta)^{op}) $$
where $\varepsilon$ is the exchange involution and $n=\deg_K(A)$ (see \cite[Proposition 2.14]{Knus1998}). Using $\varphi$ one sees $S_{\overline{k}}$ is isomorphic to a Zariski dense open subvariety in the space of hermitian elements of $A_{\overline{k}}$, which is in turn isomorphic to $\bA_{\overline{k}}^{n^2}$. Since $|k|=\infty$ , $S_{\overline{k}}$ must have a $k$-rational point, i.e. there exists $a\in A$ hermitian and with separable reduced polynomial. Let $L_1$ be the centralizer of $a$ in $A$. By the proof of \cite[Proposition 4.5.4]{Gille2006} $L_1$ is a splitting field of $A$. Since $a$ is hermitian $L_1$ is stable under $\sigma$. Therefore $L_2 = L_1^\sigma$ is a field extension of $k$ in which ${\alpha}$ is not a square and such that $A\otimes_k L_2$ is split.  Let $f: K\to k$ be the norm form of $K$ and let $\gamma$  be the unique non-trivial element of $ \Gal(K_L/L)$. By \cref{Existence of anisotropic} there exists an $n$-dimensional quadratic form $q = \langle a_1 ,\dots ,a_n\rangle$ of unsigned discriminant $1$ over a transcendental extension $L$ of $L_2$ such that $f_L\otimes q$ is anisotropic. Denote by $\sigma_q$ an involution of $M_n(L)$ adjoint to $q$. To show $\gamma\otimes_L \sigma_q$ is anisotropic we use the correspondence between unitary involutions and hermitian forms, see \cite[Proposition 2.20]{Knus1998} for details. One easily checks that the following hermitian form on $K_L^n$ is adjoint to $\gamma \otimes_L \sigma_q$:
$$h((x_i)^n_{i=1},(y_j)^n_{j=1}) = \sum a_i x_i \gamma(y_i). $$
For all $x\in K_L^n$ we have $h(x,x) = (f_L\otimes_L q)(x)$. Since $f_L\otimes_L q$ is anisotropic this implies $h$ and $\gamma\otimes_L \sigma_q$ are anisotropic. Since $L_2\subset L$ is transcendental $\alpha$ is not a square in $L$ and $A_L$ is split. 
\end{proof}

We recall some notation defined in \cite[Chapter V]{Lam2004}. Let $q$ be a quadratic form defined over a field $k$. We denote by $C(q)$ the corresponding Clifford algebra. Define $c(q) = [C(q)]$ if $\dim(q)$ is even and $c(q) = [C_0(q)]$ if it is odd. Here $[\cdot]$ denotes Brauer equivalence classes and $C_0$ is the even part of $C(q)$. $c(q)$ is called the Witt invariant of $q$. To deal with simply connected groups of type $D_5$ we will need to characterize isotropy of $\Spin(q)$-torsors where $q$ is a quadratic form. We will do this using the natural map $H^1(k,\Spin(q))\to H^1(k,\SO(q))$  whose image consists of quadratic forms with the same Witt invariant, discriminant and dimension as $q$ by \cite[III.3.2.b, Page 140]{Serre1997}. The following lemma collects some well-known properties of the Witt invariant. Proofs can be found in \cite[Chapter 5, 3.15,3.16]{Lam2004} and \cite[Chapter 5, Corollary 3.3]{Lam2004}. 

\begin{lemma}\label{properties of Cliff invar}
The following properties of the Witt invariant hold.
\begin{enumerate}
    \item If $q_1$ and $q_2$ be even dimensional quadratic forms over $k$. We have: $$c(q_1 \perp q_2) = c(q_1) + c(q_2) + [(d_{\pm}(q_1),d_{\pm}(q_2))_k],$$ where $d_{\pm}(\cdot)$ denotes the signed determinant.
    \item For any $a,b \in k:$ $$c(\langle 1,-a,-b,ab\rangle)= [(a,b)_k].$$
    \item For any odd dimensional form $q$ and $a\in k^*$ we have:
    $$c(\langle a\rangle q) = c(q). $$
    \item For any even dimensional form $q$ and $a\in k^*$ we have:
    $$c(\langle a\rangle q) = c(q)+ [(a,d_{\pm}(q))_k]. $$
\end{enumerate}
\end{lemma}

We now prove the existence of $10$-dimensional anisotropic forms with a given Witt invariant and discriminant, excluding the case of trivial discriminant and split Witt invariant. Part two of the following proof is a quadratic form theoretic version of the case $n=10$ in the proof of \cite[Proposition 6]{Tits1990}.

\begin{proposition}\label{existence split isotropic clifford}
Let $q$ be a quadratic form of dimension ten. Let $C$ be the Clifford algebra of $q$. If $C_0$ is not split then there exists an anisotropic quadratic form $q'$ over a field extension $k\subset K$ such that $q_K$ has the same Witt invariant, discriminant and dimension as $q'$.
\end{proposition}

\begin{proof}
Let $\delta = d_{\pm}(q)$. We will construct $q$ in three stages. 
\begin{enumerate}

    \item By \cite[Chapter V, Theorem 2.5]{Lam2004} $C_0$ is simple and $Z(C_0)= k(\sqrt{\delta})$. Denote $F= k(\sqrt{\delta})$. We have isomorphisms of $F$-algebras: 
    \begin{align*}
        C_{0,F} &= C_0 \otimes_k F \\
                            &= C_0 \otimes_F ( F \otimes_k F) \\
                            &= C_0 \otimes_F (F \oplus F) = C_0 \oplus C_0.
    \end{align*}
    It follows that $C_F$ is not split by part $(3)$ of \cite[Chapter V, Theorem 2.5]{Lam2004} and Wedderburn's theorem. Therefore extending scalars to $F$ we may assume $C$ is not split and $\delta$ is a square.
    
    \item Assume that $C$ is Brauer equivalent to a quaternion division algebra $Q=(a,b)$. Define the following central division algebra over the field extension $K'=k(t_1,t_2)$:
    $$ E := Q_{K'} \otimes_{K'} ( t_1, t_2)_{K'}.$$
    By \cite[Theorem 4.8, Chapter III]{Lam2004} $E$ is a division algebra with an associated anisotropic Albert form:
    $$ q_0 = q_{Q_{K'}} \perp \langle -1\rangle q_{(t_1,t_2)_{K'}}.$$
    Here for any quaternion algebra $D$, $q_D$ denotes the restriction of the norm of $D$ to the subspace of pure quaternions. Let $I$ denote the fundamental ideal in the Witt ring $W(K')$. By \cref{properties of Cliff invar} the Witt invariant determines a homomorphism $c: I^2 \to \Br(K')$. Since in $I^2$, $q_0$ is equal to the difference of the norm forms of $Q_{K'}$ and $(t_1,t_2)_{K'}$ we have:
    $$c(q_0)=[Q_{K'}] + [(t_1, t_2)_{K'}].$$ 
    The form 
    $$q_1 := \langle 1 , -t_1, -t_2, t_1 t_2\rangle $$
    has Witt invariant $[(t_1, t_2)_{K'}]$ and trivial signed discriminant by \cref{properties of Cliff invar}. We can now define $K= k(t_1,t_2,t_3)$ and the sought after form $q'$:
    $$q' := q_0 \perp t_3  q_1.$$
    By \cite[Exercise 1, Chapter IX]{Lam2004} $q'$ is anisotropic. Using \cref{properties of Cliff invar} again we see:
    \begin{align*}
        c(q') &= c(q_0) + c(q_1) + [( d_{\pm}(q_0), d_{\pm}(q_1))_K]\\
              &= [Q_K].
    \end{align*}
    Finally $d(q')= d(q_0)d(q_1)=-1$ so $d_{\pm}(q')=1$ as required.
    \item Assume that $C$ is an arbitrary central simple algebra of exponent two and $C$ is not split.  Let $X$ be the generalized Severi-Brauer variety of right ideals of $C$ of dimension $2\deg(C)$. By \cite[Theorem 3]{Blanchet1991} the index of $C_{k(X)}$ is two. Therefore  $C_{k(X)}$ is Brauer equivalent to some quaternion division algebra with center $k(X)$. We can now proceed as we did in stage two.
    
\end{enumerate}
\end{proof}

\begin{proposition}\label{existence of anisotropic disc}
Let $k$ be a field. Let $\delta\in k^*$ be an element which is not a square. For any $a\in k^*$ there exists a ten dimensional quadratic form over a field extension $k\subset L$ which has Witt invariant $(a,\delta)_L$ and discriminant $\delta$.
\end{proposition}
\begin{proof}
Define a $k$-algebra $B=M_4(K)$ where $K = k(\sqrt{\delta})$. 
By \cref{Anisotropic hermitian involution} there exists an anisotropic 4-dimensional quadratic form $q_0$ over a transcendental field extension $k\subset L$ with trivial discriminant such that $\delta$ is not a square in $L$ and $\tau = \gamma \otimes_L \sigma_q$ is anisotropic where $\gamma$ is a unitary involution on $K$ and $\sigma_q$ is an orthogonal involution of $M_n(L)$ adjoint to $q$. By \cite[lemma 10.33]{Knus1998} the discriminant algebra $D(B_L,\tau)$ is split. By \cite[Theorem 15.24]{Knus1998} the canonical orthogonal involution on the discriminant algebra $D(B_L,\tau)$ has Clifford algebra isomorphic to $B_L$. Let $q$ be a quadratic form adjoint to this orthogonal involution.  By \cite[Proposition 15.39]{Knus1998} $q$ is anisotropic because $\tau$ is. Let $C(q)$ be its Clifford algebra. The even part of $C(q)$ is split since it is isomorphic to $B_L$. By \cite[Theorem 15.24]{Knus1998} we have $\disc(q) = \delta$ and \cite[Theorem 2.5, Chapter V]{Lam2004} implies $C(q)$ is split by $L(\sqrt{\delta})$. Therefore $C(q)$ is Brauer equivalent to a quaternion algebra $Q$ over $L$. We now separate into two cases:
\begin{enumerate}
    \item Assume $(a,\delta)_L$ is split. We want choose $q$ such that $Q$ is a division algebra. If $Q$ is split, scale $q$ by an element $t\in L$ such that $(t,\delta)_k$ is non-split to assume $Q$ is a division algebra using \cref{properties of Cliff invar} (we might have to enlarge the field to find such an element $t$).  Let $n(Q)$ be the norm form of $Q$.  We define a quadratic form over $L(z)$:
$$q_1 = q \perp z n(Q). $$
By \cite[Chapter IX, Exercise 1]{Lam2004} $q_1$ is anisotropic. By \cref{properties of Cliff invar} we have:
$$c(q_1) =  c(q) + c(n(Q)) = 2[Q] =0.$$
The discriminant of $q_1$ is $\delta$, same as the discriminant of $q$ and so we are done.
    
    \item  Assume $(a,\delta)_L$ is a division algebra. By \cite[Theorem 4.1, Chapter III]{Lam2004}, $Q = (b,\delta)_L$ for some $b\in k^*$. Scaling $q$ by $b$ we may assume $Q$ is split using \cref{properties of Cliff invar}.  Let $n(a,\delta)_L$ be the norm form of $(a,\delta)_L$.  We define a quadratic form over $L(z)$:
$$q_1 = q \perp z n(a,\delta)_L. $$
By \cite[Chapter IX, Exercise 1]{Lam2004} $q_1$ is anisotropic. By \cref{properties of Cliff invar} we have:
$$c(q_1) =  c(q) + c(n(a,\delta)_L) = [(a,\delta)_L].$$
Clearly the discriminant of $q_1$ is $\delta$ and so we are done.
\end{enumerate}
\end{proof}

We finish this section with a lemma showing the existence of division algebras with certain Schur index properties. It is a strengthening of \cite[Lemma 5.1]{Merkurjev1996}. 
\begin{lemma}\label{existence of division}
Let $k$ be a field and let $d\mid n$ be positive integers. Assume the square free part of $n$ divides $d$. Then there exists a central division algebra $\epsilon$ of index $n$ and exponent $d$ over a field extension $k\subset K$ such that $D_K\otimes_K \epsilon$ is a division algebra for any central division algebra $D$ over $k$.
\end{lemma}
\begin{proof}
Throughout the proof let $D$ be an arbitrary central division algebra over $k$. By the primary decomposition theorem the proof immediately reduces to the case, where $n = p^r$  and $d = p^s$ for some prime p and some integers $1 \leq s \leq r$. Assume $n=p^r$, $d=p^s$ for some $1\leq s\leq r$. Let $F=k(t_1,\dots,t_{p^s})$ and denote $K_0 = F^\sigma$ where $\sigma$ is the automorphism permuting $t_1,\dots,t_{p^s}$ cyclically. $F/K_0$ is a cyclic Galois extension of degree $p^s$. We proceed by induction on $r-s$. If $r-s=0$, let $K = K_0(t)$ and denote by $\epsilon$ the cyclic algebra $(F(t)/K, \sigma, t)$ as in \cite[Lemma 5.1]{Merkurjev1996}. By \cite[Lemma 5.1]{Merkurjev1996}, $D_K\otimes \epsilon$ is a division algebra and $\exp(\epsilon)=\ind(\epsilon)=p^s$. For the induction step assume $r-s\geq 1$ and let $\epsilon_0$ be a division algebra with index $p^{r-1}$ and exponent $p^s$ over a field extension $K_1$ such that $D_{K_1}\otimes \epsilon_0$ is a division algebra. By the base case we can find a division algebra $\epsilon_1$ over a field extension $K_1\subset K$ of index and exponent $p$ such that $D_K\otimes \epsilon_{0,K}\otimes \epsilon_1$ is a division algebra. Put $\epsilon = \epsilon_{0,K}\otimes \epsilon_1$. By assumption $\epsilon$ is a division algebra of index $p^r$ since $\epsilon\otimes D_K$ is a division algebra and the exponent of $\epsilon$ is $p^s$.  This finishes the proof of the lemma.
\end{proof}


\section{Classification of simple strongly isotropic algebraic groups}\label{Classification of simple strongly isotropic algebraic groups}
In this section we prove \cref{simple strongly isotropic}. By \cref{simply connected split groups} we only have to consider three families of classical groups.

\begin{lemma}\label{strongly isotropic families}
A strongly isotropic simple group $G$ is of type $A_n$, $C_n$ or $D_5$.
\end{lemma}
\begin{proof}
By definition $ G_{\overline{k} }$ is strongly isotropic and split. By \cref{covering thm1} the universal cover of $G_{\overline{k}}$ is strongly isotropic as well and therefore it is of type $A_n$, $C_n$ or $D_5$ by \cref{simply connected split groups}.  Thus $G$ is of one of these types as well.
\end{proof}

We break the proof of \cref{simple strongly isotropic} into cases according to the type of $G$. These are contained in the subsections below. In each case we show that $G$ is strongly isotropic if and only if it is isomorphic to one of the groups listed in \cref{simple strongly isotropic}. Our main tool throughout the proof of \cref{simple strongly isotropic} is the analysis of twisted flag varieties of classical simple groups carried out in \cite{Merkurjev1996} and \cite{Merkurjev1998}. These are connected to the question of isotropy by \cref{isotropy and flags}.

\subsection{Type \texorpdfstring{$A_{l}$}{A}}
Throughout this subsection let $G$ be a simple group of type $A_l$. 

\begin{lemma}
If $G$ is of outer type then it is not strongly isotropic.
\end{lemma}
\begin{proof}
By \cite[Theorem 26.9]{Knus1998} if $G$ is of outer type then it is the special unitary group of a simple algebra with involution $(A,\sigma)$ with center $Z(A) = K$ where $k\subset K$ is a Galois extension of order $2$. Let $\tilde{G} = \text{GU}(A,\sigma)$ be the general unitary group. Since $G$ is the derived group of $\tilde{G}$ it suffices to prove $\tilde{G}$ is not strongly isotropic by \cref{reduction to semisimple}. By  \cref{Anisotropic hermitian involution} there exists an anisotropic unitary involution $\tau$ on $A_L$ for some field extension $k\subset L$ in which $\alpha$ is not a square. Let $[\gamma]\in H^1(L,\tilde{G})$ be the torsor corresponding to $\tau$ as in \cite[Page 401]{Knus1998}. For any parabolic subgroup $P\subset \tilde{G}$ the twisted flag variety ${}_\gamma (\tilde{G}/P) = {}_\gamma (G/P\cap G)$ consists of flags of $\tau$-isotropic ideals of $A_L$ (see \cite[Page 172]{Merkurjev1998}) and therefore has no $K$-points since $\tau$ is anisotropic. Therefore $[\gamma]$ is anisotropic by \cref{isotropy and flags}.
\end{proof}

It remains to consider groups of inner type. By the classification of simple classical groups these are of the form $\SL_n(D)/\mu_d$ for some integers $d\mid n\deg(D)$ and a central division algebra $D$ over $k$. By \cref{reduction to semisimple}  $\SL_n(D)/\mu_d$ is strongly isotropic if and only if $\GL{n}(D)/\mu_d$ is. For the rest of the subsection fix $G = \GL{n}(D)/\mu_d$. Our goal is to show that the square free part of $n$ divides $d$ if and only if $G$ is not strongly isotropic. The following lemma gives us a grasp on the Galois cohomology of $G$ using the natural projection $\Ad: G\to \overline{G}= \PGL_n(D)$. Recall the cohomology set $H^1(k,\overline{G})$ classifies central simple algebras of degree $n\ind(D)$.

\begin{lemma}\label{not simply connected An}
  Let $A$ be a central simple algebra of degree $n\ind(D)$. Let $[\gamma]\in H^1(k,\overline{G})$ be the corresponding torsor.  $[\gamma]$ lies in the image of natural map $ H^1(k,G)\to H^1(k,\overline{G})$ if and only if $d([A]-[D])=0$ in the Brauer group of $k$.
\end{lemma}
 \begin{proof}
 The case where $D$ is split was proven in \cite[Lemma 2.6]{Berhuy2005}. The general case follows from the split case by twisting the exact sequence $1\to \Gm \to \GL{n\ind(D)} \to \PGL_{n\ind(D)}\to 1$ by the cohomology class of $M_n(D)$ in $H^1(k,\PGL_{n\ind(D)})$ using \cite[Proposition I.44]{Serre1997} and \cite[Proposition I.35]{Serre1997}.
\end{proof}

Our next goal is to characterize isotropy of $\overline{G}$-torsors.

\begin{lemma}\label{criterion for isotropy An}
A torsor $[\gamma] \in H^1(k,\PGL_n(D))$  corresponding to a central simple algebra $A$ is anisotropic if and only if $$\ind(A) =  \gcd(\ind(A),\ind(D)) n.$$
In particular if $D = k$, $[\gamma]$ is anisotropic if and only if $A$ is a division algebra.
\end{lemma}
\begin{proof}
We use the notation and results of \cite[Page 561]{Merkurjev1996}. Denote $\ind(D)= r$. Any parabolic subgroup of $P\subset \overline{G}$ is of type $(r n_1,\dots,r n_t)$ for some integers $1\leq n_1 \leq \dots\leq n_t\leq n$. The twisted flag variety ${}_\gamma (G/P)$ is the variety of flags of ideals $$ I_1\subset\dots \subset I_t\subset A, \text{ for all }j:\ \dim_k I_j = n r^2 n_j.$$ Therefore by \cref{isotropy and flags} $[\gamma]$ is isotropic if and only if $A$ has an ideal of dimension $n r^2m_1$ for some $m_1\lneq r n$. On the other hand the ideals of $A$ are of dimension $n r\ind(A) m_2,$ for $0\leq m_2 \leq \frac{n r}{\ind(A)}$. Therefore $A$ is anisotropic if and only if:
$$\frac{\ind(A)r}{\gcd(\ind(A),r)} = \text{lcm}(\ind(A),r) = n r.$$
Since $r=\ind(D)$ this is what we wanted to prove. 
\end{proof}

To prove the existence of anisotropic $G$-torsors we will need to know division algebras with certain Schur index properties exist.

We can now put it all together and finish the case of type $A_{l}$. Recall that we set $G = \GL{n}(D)/\mu_d$ and our goal is to prove $G$ is strongly isotropic if and only if there exists a prime divisor of $n$ which does not divide $d$.

\begin{proposition}
$G$ is strongly isotropic if and only if there exists a prime divisor $p$ of $n$ which does not divide $d$.
\end{proposition}
\begin{proof}
\begin{itemize}
    \item Denote by $t$ the square free part of $n$. We assume $t$ divides $d$ and show $G$ is not strongly isotropic. By \cref{existence of division} there exists a field extension $k\subset K$ and a central division algebra $\epsilon$ over $K$ with index $n$ and period $t$ such that:
    $$\ind(\epsilon\otimes_K D_K) = n \ind(D_K).$$
    Set $A = \epsilon\otimes_K D_K$ and let $[\gamma]\in H^1(K,\overline{G})$ be the corresponding torsor. By assumption $t\mid d$ and so $(A\otimes_K D^{op}_K)^d$ is split. Therefore $[\gamma]$ lies in the image of natural map $ H^1(k,G)\to H^1(k,\overline{G})$ by \cref{not simply connected An}. On the other hand \cref{criterion for isotropy An} implies $[\gamma]$ is anisotropic. Therefore $G$ is not strongly isotropic
    \item Let $p$ be a prime divisor of $n$ which does not divide $d$. To show $G$ is strongly isotropic we assume it admits anisotropic torsors over a field extension $K$ and derive a contradiction. Assume $[\gamma]$ is an anisotropic torsor in $H^1(K,G)$ corresponding to a central simple algebra $A$. By \cref{not simply connected An}  $A^d$ is Brauer equivalent to $D_K^d$. Therefore:
    $$\ind(A^d) = \ind(D_K^d)\mid \ind(D_K).$$
    Since $[\gamma]$ is anisotropic,  \cref{criterion for isotropy An} implies:
    $$\ind(A) = \gcd(\ind(A),\ind(D_K))n.$$
    Let $r$ be the largest exponent such that $p^r$ divides $\gcd(\ind(A),\ind(D_K))$. By assumption $p\mid n$ and so $p^{r+1}$ divides $\ind(A)$. Since $p$ does not divide $d$, $p^{r+1}$ divides $\ind(A^d)=\ind(D_K^d)$ and so $p^{r+1}$ divides $\ind(D_K)$, see \cite[Theorems 5.5, 5.7]{Saltman1999}. Therefore $p^{r+1}\mid \gcd(\ind(A),\ind(D_K))$. This contradicts maximality of $r$ and finishes the proof.
\end{itemize}  

\end{proof}

\subsection{Type \texorpdfstring{$C_l$}{C}}

Let $G$ be a simple group of type $C_n$. We show  $G$ is strongly isotropic if and only if $G = \Sp_{2n}$. One direction is easy: since $H^1(K, \Sp_{2n}) = 1$ for every $K/k$, $\Sp_{2n}$ is obviously strongly isotropic. Our goal is thus to prove the opposite implication: If $G$ is strongly isotropic, then it must be isomorphic to $\Sp_{2n}$. 

\begin{lemma}\label{Cn simply connected}
If $G$ is not simply connected, then it is not strongly isotropic.
\end{lemma}
\begin{proof}
Since $|Z(\Sp_{2n})|=2$ if $G$ is not simply connected it is adjoint. By \cite[Theorem 26.14]{Knus1998} and \cite[29.22, Page 404]{Knus1998}, $G$-torsors correspond to isomorphism classes of central simple algebras of degree $2n$ with symplectic involution $(A,\tau)$. By \cref{anisotropic symplectic involution} after passing to a field extension we may choose $(A,\tau)$ with $\tau$ anisotropic. Then for any parabolic $P\subset G$ the twisted flag variety ${}_\gamma (G/P)$ has no $k$-point since it parametrises flags of $\tau$-isotropic ideals of $A$ by \cite[5.24, Page 53]{Merkurjev1996}. Therefore $[\gamma]$ is anisotropic by \cref{isotropy and flags}.
\end{proof}

The next proposition finishes the case of groups of type $C_n$.

\begin{proposition}
If $G$ is strongly isotropic then $G\cong \Sp_{2n}$.
\end{proposition}
\begin{proof}
By \cref{Cn simply connected} and \cite[Theorem 26.14]{Knus1998}, $G$ is the symplectic group of some central simple algebra $A$ with a symplectic involution $\sigma$. Since all symplectic involutions on $M_n(k)$ are isomorphic it suffices to check $A$ must be split. Assume the contrary: $A$ is not split. Let $\tilde{G} = \text{GSp}(A,\sigma)$ be the group of symplectic similitudes. Since $G$ is the derived group of $\tilde{G}$ it suffices to prove $\tilde{G}$ is not strongly isotropic by \cref{reduction to semisimple}. $\tilde{G}$-torsors correspond to conjugacy classes of symplectic involutions on $A$ (see \cite[29.23]{Knus1998}). By \cref{anisotropic symplectic involution} after extending scalars to a field extension $K$ we may assume there exists an anisotropic symplectic involution $\tau$ on $A$. Let $[\gamma]\in H^1(K,\tilde{G})$ be a torsor corresponding to $\tau$. For any parabolic $P\subset \tilde{G}$ the twisted flag variety ${}_\gamma (\tilde{G}/P)= {}_\gamma (G/P\cap G)$ has no $K$-point since it parametrises flags of $\tau$-isotropic ideals of $A$ by \cite[5.24, Page 53]{Merkurjev1996}. This contradicts our assumption that $G$ is strongly isotropic by \cref{isotropy and flags} and therefore we conclude $A$ must be split.
\end{proof}

\subsection{Type \texorpdfstring{$D_5$}{D}}
Let $G$ be a simple group of type $D_5$. If $G = \Spin(q)$ for some regular 10-dimensional quadratic form with trivial discriminant and split Clifford algebra then $G$ is strongly isotropic by \cref{simply connected split groups} because $G$ is a strong inner form of $\Spin_{10}$ (see \cite[Page 140, III.3.2 Example b]{Serre1997}). It remains to prove that if $G$ is strongly isotropic then $G = \Spin(q)$ for some 10-dimensional quadratic form with trivial discriminant and split Clifford algebra. Let $q$ be a (regular) quadratic form. Torsors $[\gamma] \in H^1(k,\SO(q))$ correspond to quadratic forms $q'$ with the same discriminant as $q$, see \cite[29.29]{Knus1998}. The group $\SO(q)$ is isotropic if and only if $q$ is isotropic if and only if $\SO(q)$ has a parabolic which is the stabilizer of a one dimensional completely isotropic subspace of $k^{\dim(q)}$ \cite[5.49]{Merkurjev1996}.

\begin{lemma}\label{criterion for isotropy D5}
Let $q$ be an isotropic quadratic form. A torsor $[\gamma] \in H^1(k,\SO(q))$ corresponding to a quadratic form $q'$ is isotropic if and only if $q'$ is isotropic.
\end{lemma}
\begin{proof}
Since $q$ is isotropic, $\SO(q)$ has a parabolic $P$ which is the stabilizer of a one dimensional completely isotropic subspace of $k^{\dim(q)}$. By \cref{isotropy and flags} $[\gamma]$ is isotropic if and only if ${}_\gamma \SO(q) = \SO(q')$ has a parabolic of the same type as $P$ if and only if $\SO(q')$ is isotropic. Therefore by the remark preceding the lemma $[\gamma]$ is isotropic if and only if $q'$ is isotropic.
\end{proof}

\begin{lemma}\label{D5 simply connected}
If $G$ is not simply connected then it is not strongly isotropic.
\end{lemma}
\begin{proof}
As in the proof of \cref{strongly isotropic families} by extending scalars to the algebraic closure we may assume $G$ is split and $\sqrt{-1}\in k$. In that case if $G$ is not simply connected it is covered by $\SO(q_0)$ where $q_0$ is a 10-dimensional hyperbolic quadratic form. By \cref{covering thm1} it suffices to show $\SO(q_0)$ admits anisotropic torsors. Since $\sqrt{-1}\in k$, by \cref{Existence of anisotropic} there exists a 10-dimensional anisotropic form $q$ with trivial discriminant over a field extension $k\subset K$. Since $\disc(q_0) = 1$, $q$ corresponds to an anisotropic torsor in $H^1(k,\SO(q_0))$ by \cref{criterion for isotropy D5}.
\end{proof}

The next proposition finishes the case of type $D_5$ and therefore the proof of \cref{simple strongly isotropic}.

\begin{proposition}
If $G$ is strongly isotropic then $G = \Spin(q)$ for some regular ten dimensional quadratic form with trivial discriminant and split Clifford algebra.
\end{proposition}
\begin{proof}
By \cite[Theorem 26.15]{Knus1998} and \cref{D5 simply connected} $G$ is isomorphic to $\Spin(A,\sigma)$ for some central simple algebra with orthogonal involution over $k$. Denote by $C$ the Clifford algebra of $(A,\sigma)$ as defined in \cite[Section 8.B]{Knus1998} and let $Z = Z(C)$.
We need to show $A$ and $C$ are split and $\disc(\sigma)=1$. We break the proof into three steps.
    \begin{enumerate}
        \item Assume the Brauer class $[C]$ is not in the cyclic group generated by $[A_Z]$. By extending scalars to a field extension $k\subset K$ we may split $A$ without splitting $C$ by Amitsur's Theorem (see \cite[Theorem 5.4.1]{Gille2006}). The involution on $A_K$ is adjoint to some 10-dimensional quadratic form $q$. By functoriality of the Clifford algebra and \cite[Proposition 8.8]{Knus1998} the even part of the Clifford algebra of $q$ is isomorphic to $C_K$ and thus is not split. Therefore by  \cref{existence split isotropic clifford} after possibly enlarging $K$ there exists an anisotropic 10-dimensional form $q'$ with the same Witt invariant and discriminant as $q$. Let $[\gamma]\in H^1(K,\SO(q))$ be the anisotropic torsor corresponding to $q'$. By \cite[Page 140 ,III.3.2 Example b]{Serre1997}, $[\gamma]$ lies in the image of the natural map $H^1(K,\Spin(q))\to H^1(K,\SO(q))$. This contradicts our assumption that $G$ is strongly isotropic.
        
        \item Therefore $[C]$ is in the cyclic group generated by the Brauer class $[A_Z]$. For some $n\in \bN$ we have $[C] = n[A_Z]$.  By \cite[Theorem 9.12]{Knus1998}, $C\otimes C$ is Brauer equivalent to $A_Z$. Therefore:
        $$2n[A_Z] = 2[C] = [A_Z].$$
        Since $A$ has exponent two it follows that $A_Z$ is split.  Therefore $$[C] = n[A_Z] =0 .$$
        If $Z = k\times k$ then $A$ is split and $\disc(\sigma)$ is a square by \cite[Theorem 8.10]{Knus1998}. We assume this is not the case (i.e. $Z=k(\sqrt{\disc(\sigma)})$ is a quadratic field extension of $k$) and reach a contradiction. 
        
        \item Let $K = k(SB(A))$ be the function field of the Severi-Brauer variety of $A$. Since $SB(A)$ is absolutely irreducible, $k$ is integrablly closed in $K$ and $\disc(\sigma)$ is not a square in $K$. Therefore by extending scalars to $K$ we may assume $A$ is split without loss of generality. Then $G$ is the spin group of some 10-dimensional quadratic form $q$ with non-trivial discriminant since $\disc(q) = \disc(\sigma)$. The even part of the Clifford algebra $C(q)$ is split since it is isomorphic to $C$. Therefore by \cite[Theorem 2.5, Chapter V]{Lam2004}, $C(q)_Z$ is split. We see that $C(q)_Z$ is Brauer equivalent to a quaternion algebra $(a, \disc(\sigma))_k$ for some $a\in k$ by \cite[Theorem 4.1, Chapter III]{Lam2004}. By \cref{existence of anisotropic disc} there exists an anisotropic 10-dimensional quadratic form $q'$ over a field extension $k\subset K$ with the same Witt invariant and discriminant as $q$. Let $[\gamma]\in H^1(K,\SO(q))$ be the anisotropic torsor corresponding to $q'$. By \cref{criterion for isotropy D5} $[\gamma]$ is anisotropic. By \cite[III.3.2.b, Page 140]{Serre1997}, $[\gamma]$ lies in the image of the natural map $H^1(K,\Spin(q))\to H^1(K,\SO(q))$. This contradicts our assumption that $G$ is strongly isotropic.

    \end{enumerate} 

\end{proof}


\section{A partial classification of strongly isotropic semisimple groups}\label{semisimple section}

Let $G$ be a semisimple group over $k$. When $G$ splits as a direct product of simple groups, we can use \cref{simple strongly isotropic} and \cref{products and strong isotropy} to check if $G$ is strongly isotropic. However in general $G$ is only isogenous to a product of simple groups and as we have seen, strong isotropy is not preserved by isogenies. Consequently, applying \cref{simple strongly isotropic} to check if $G$ is strongly isotropic requires a more careful analysis. In this section we carry out this analysis under certain restrictions on the root system of $G$. In the next section we prove another partial classification result which illustrates the difficulties that arise when we lift these restrictions. Our goal is to prove \cref{strongly isotropic quotient} which we now restate for the readers convenience. 

\begin{theorem}
  Let $G$ be a semisimple group over a field $k$ with $\Char(k)\neq 2$. Assume for any simple factor of $G$ of type $A_{n-1}$ the integer $n$ is squarefree. Then $G$ is strongly isotropic if and only if it admits a simple strongly isotropic quotient.
\end{theorem}

If $G$ admits a simple strongly isotropic quotient then it is strongly isotropic by \cref{covering thm1}. For the rest of this section we assume $G$ is a strongly isotropic group which satisfies the hypothesis of \cref{strongly isotropic quotient} and prove that it admits a simple strongly isotropic quotient. Let $\pi: \tilde{G} \to G$ be the universal cover of $G$. Write $\tilde{G} = \prod_i \tilde{G}_i$ for some simply connected simple groups $\tilde{G}_1, \dots, \tilde{G_r}$  and denote by $Z=\ker\pi$ the fundamental group of $G$. Note that $Z$ may not be reduced. We denote the projection onto the $i$-th coordinate from $\tilde{G}$ by $p_i$. By \cref{versal torsors} $G$ admits reduction of structure to some maximal parabolic $P\subset G$. Since $\pi$ is central, \cite[Theorem V.22.6]{Borel1991} implies $P = \pi(\tilde{P})$ for some maximal parabolic of $\tilde{G}$. By \cref{products and strong isotropy} we may assume without loss of generality that $\tilde{P}$ is of the form:
$$\tilde{P} = \tilde{P}_1 \times \tilde{G}_2 \times\dots\times \tilde{G}_r \ , \ \ \text{where } \tilde{P}_1 \ \text{is a maximal proper parabolic of } \tilde{G}_1. $$
Let $G_1 = \tilde{G}_1 / p_1(Z)$ and denote by $\pi_1$ the canonical projection $G\to G_1$, $P_1 = \pi_1(P)$. We denote by $\overline{G_i}$ the adjoint group of $\tilde{G_i}$ for all $i$. Our goal is to prove that $G_1$ is strongly isotropic. The following lemma connects our assumptions to \cref{simple strongly isotropic}.

\begin{lemma}\label{P strongly isotropic}
Let $U \in H^1(K_0,G)$ be a versal torsor over a field extension $k\subset K_0$. Since $G$ admits reduction of structure to $P$, $q_*U = [\gamma_1]$ for some $[\gamma_1]\in H^1(K_0,P_1)$ by \cref{covering thm1}. The group ${}_{\gamma_1}\tilde{G}_1$ is strongly isotropic.
\end{lemma}
\begin{proof}
For all $i\geq 2$ denote by $[\gamma_i]$ the image of $U$ under the natural map $H^1(K_0,G) \to H^1(K_0, \overline{G_i})$. Let $K_0\subset K$ be a field extension. By \cite[Proposition I.35]{Serre1997}, ${}_\gamma G$ admits reduction of structure to ${}_{\gamma}P$. \cref{covering thm1} implies ${}_\gamma \tilde{G}$ admits reduction of structure to ${}_\gamma \tilde{P}$ and so $H^1(K,{}_\gamma \tilde{P})\to H^1(K,{}_\gamma \tilde{G})$ is surjective. Since these cohomology sets split as products:
\begin{align*}
    H^1(K,{}_\gamma \tilde{P}) &= H^1(K,{}_{\gamma_1}\tilde{P_1})\times\prod_{i=2,\dots ,r}H^1(K,{}_{\gamma_i}\tilde{G_i})\\   H^1(K,{}_\gamma \tilde{G}) &=\prod_{i=1,\dots ,r}H^1(K,{}_{\gamma_i}\tilde{G_i}),
\end{align*}
We conclude that 
$H^1(K,{}_{\gamma_1}P_1)\to H^1(K,{}_{\gamma_1}G_1)$ is surjective.
\end{proof}

Fix $U,K_0, [\gamma_1]$ as in the statement of \cref{Br invariants 1} for the rest of this section. Using \cref{P strongly isotropic} and our classification of strongly isotropic simply groups we will deduce that certain invariants of $[\gamma_1]$ vanish. The invariants we use are called \emph{normalized Brauer invariants}. Let $A$ be an algebraic group over $k$. Denote by $H^1(\ast,A)$ the functor taking a field extension $k\subset K$ to the \emph{pointed set} $H^1(K,A)$. A normalized Brauer invariant of $A$-torsors is a morphism of functors between $H^1(\ast,A)$ and the Brauer group functor $K\mapsto \Br(K)$. The adjective "normalized" refers to the functor preserving the pointed set structure. From now on all invariants are assumed to be normalized. The set of Brauer invariants of $A$-torsors is denoted $\Inv(A,\Br)$. It comes with a natural abelian group structure of pointwise multiplication in the Brauer group. Any morphism of algebraic groups $f:A\to B$ induces a homomorphism $f^\#:\Inv(B,\Br) \to \Inv(A,\Br)$ given by precomposition with the induced functor $f_*:H^1(\ast,A)\to H^1(\ast,B)$. The following fact makes Brauer invariants useful for studying the relationship between $G$-torsors and $G_1$-torsors.

\begin{lemma}\label{Br invariants 1}
The induced homomorphism on Brauer invariants $\pi_1^\#:\Inv(G_1,\Br)\to \Inv(G,\Br)$ is injective.
\end{lemma}

\begin{proof}
Let $Z^*$, $p_1(Z)^*$ be the group of characters defined over $k$ of $Z$ and $p_1(Z)$ respectively. By \cite[Example 31.20]{Knus1998} we have a commutative diagram (see also \cite[Theorem 2.4]{Blinstein2013}, \cite[Proposition 6.10]{Sansuc1981} for more details):
$$
\begin{tikzcd}
Z^* \arrow[rr, "a_G"]                      &  & {\Inv(G,\Br)}                   \\
p_1(Z)^* \arrow[rr, "a_{G_1}"] \arrow[u, "f"'] &  & {\Inv(G_1,\Br)} \arrow[u, "{\pi_1}^\#"]
\end{tikzcd} $$
Here $f$ is the pull back map induced from $p_1$ which is injective since ${p_1}_{|Z}:Z\to p_1(Z)$ is surjective. Since $a_G$ and $a_{G_1}$ are isomorphisms it follows that $\pi_1^\#$ is injective.
\end{proof}

The following lemma follows directly from \cite[Theorem 2.2]{Blinstein2013} and the proof of \cite[lemma 5.4.6]{Gille2006}. Recall $U\in H^1(K_0,G)$ denotes a versal $G$-torsor.
\begin{lemma}\label{Br invariants 2}
The evaluation homomorphism $\theta: \Inv(G,\Br) \to \Br(K_0), f\mapsto f(K_0)(U)$ is injective.
\end{lemma}

We now have all the ingredients needed to finish the proof of \cref{strongly isotropic quotient}.
\begin{proposition}
The group $G_1$ is strongly isotropic.
\end{proposition}
\begin{proof}
We proceed by separating into cases according to the type of $G_1$. In each case we characterize isotropy of torsors in $H^1(K,\overline{G_1})$ by the vanishing of certain Brauer invariants and prove that these invariants vanish on $\Ad_*(H^1(K,G_1))$ using strong isotropy of ${}_{\gamma_1}\tilde{G}_1$ and \cref{Br invariants 1}. Using \cref{adjoint thm1}, we conclude $G_1$ is strongly isotropic.
\begin{itemize}
    \item[Type $C_n$:] Since $\tilde{G}_1$ is strongly isotropic, \cref{simple strongly isotropic} implies it is isomorphic to $\Sp_{2n}$. By \cite[29.22]{Knus1998} torsors $[\gamma]\in H^1(K,\overline{G_1})$ classify central simple algebras of degree $2n$ with symplectic involution $(A,\sigma)$ and ${}_\gamma \tilde{G}_1$ is isomorphic to the symplectic group of $(A,\sigma)$. Let $f\in \Inv(\overline{G_1},\Br)$ be defined by $f(K)([\gamma]) = [A]$. By \cref{simple strongly isotropic}, for any $[\gamma] \in H^1(K,\overline{G_1})$ we have $f(K)([\gamma]) = 0$ if and only if ${}_\gamma \tilde{G}_1$ is strongly isotropic if and only if $[\gamma]$ is split. Since ${}_{\gamma_1} \tilde{G}_1$ is strongly isotropic \cref{Br invariants 2} implies $(\Ad\circ \pi_1)^{\#}(f) = 0$. Therefore by \cref{Br invariants 1} $\Ad^\#(f)=0$ and so $G_1$ is strongly isotropic by \cref{adjoint thm1}.
    
    \item[Type $D_5$: ]  Since $\tilde{G}_1$ is strongly isotropic, \cref{simple strongly isotropic} implies it is isomorphic to $\Spin(q)$ for some 10-dimensional quadratic form $q$ with trivial discriminant and split Clifford algebra. By \cite[Page 409]{Knus1998} torsors $[\gamma]\in H^1(K,\overline{G_1})$ classify central simple algebras of degree $10$ with orthogonal involution $(A,\sigma)$ such that $\disc(\sigma) = 1$. By \cite[Theorem 26.15]{Knus1998} ${}_\gamma \tilde{G}_1$ is isomorphic to the spin group of $(A,\sigma)$. Let $f,c_{\pm} \in \Inv(\overline{G_1},\Br)$ be defined by 
    $$f(K)([\gamma]) = [A], \ \  c_{\pm}(K)([\gamma])= C_{\pm} (A,\sigma),$$ 
    where $C_\pm (A,\sigma)$ are the components of the Clifford algebra of $A$ (see \cite[Theorem 8.10]{Knus1998} for the relevant definitions). By \cref{simple strongly isotropic}, since ${}_{\gamma_1}\tilde{G}_1$ is strongly isotropic \cref{Br invariants 2} implies $$(\Ad\circ \pi_1)^{\#}(f) = (\Ad\circ \pi_1)^{\#}(c_\pm)= 0.$$  Therefore by \cref{Br invariants 1} $$\Ad^{\#}(f) = \Ad^{\#}(c_\pm)= 0.$$ In particular torsors in $\Ad_*(H^1(K,G_1))$ correspond to pairs $(A,\sigma) =  (M_{10}(K),\sigma_q)$ where $\sigma_q$ is adjoint to a 10-dimensional quadratic form $q$ with trivial discriminant and split Clifford algebra. By \cite[Proposition 2.8]{Lam2004} such quadratic forms are always isotropic. Therefore $G_1$ is strongly isotropic by \cref{adjoint thm1} and \cref{criterion for isotropy D5}.
    
    \item[Type $A_{n-1}$:]  Since $\tilde{G}_1$ is strongly isotropic, \cref{simple strongly isotropic} implies it is isomorphic to $\SL_m(D)$ for some central division algebra $D$ with $m\neq 1$ and $m\ind(D)=n$. By assumption $m$ and $\ind(D)$ are squarefree and coprime. Torsors $[\gamma]\in H^1(K,\overline{G_1})$ classify central simple algebras $A$ with $\deg(A)= n$. For any prime divisor $p$ of $n$ define a Brauer invariant of $\overline{G_1}$-torsors: 
    $$f_p(K)([\gamma])= ([A] - [D])^{{n}/{p}}.$$ 
    By \cref{criterion for isotropy An} $[\gamma]$ is isotropic if and only if there exists a prime divisor $p$ of $m$ such that $f_p(K)([\gamma]) = 0$. $ \pi_{1,*}(U)= [\gamma_1]$ is isotropic by \cref{covering thm1}. Therefore \cref{Br invariants 2} implies that for some prime divisor of $m$ we have $(\Ad\circ \pi_1)^\#(f_p) = 0$ and by \cref{Br invariants 1} $\Ad^\#(f_p) = 0$. We conclude that $G_1$ is strongly isotropic by \cref{adjoint thm1}.
    
\end{itemize}
\end{proof}

\begin{remark}
Note that when $\tilde{G}_1$ is of type $D_5$, $C_n$ or $A_{p-1}$ with $p$ prime,  our arguments imply $G_1 = \tilde{G}_1$ or equivalently $p_1(Z)=\{1\}$. This follows from \cref{simple strongly isotropic} since any proper simple quotient of $\tilde{G}_1$ admits anisotropic torsors.
\end{remark}

\section{Strongly isotropic semisimple groups of type  \texorpdfstring{$A$}{A}}
The next theorem gives another partial classification of strongly isotropic semisimple groups. We consider groups which are of "opposite" type to the groups covered by the previous theorem. We allow only factors of type $A_n$ and without any restriction on $n$.  The complexities that arise when considering factors of type $A_{n-1}$ with $n$ not squarefree are illustrated well by the complicated divisibility criterion we obtain. To make the computations easier we work with products of copies of $\GL{n}$ instead of $\SL_n$. \cref{reduction to semisimple} shows this makes no difference as far as strong isotropy is concerned. 
Let $n_1,\dots,n_r$ be positive integers and let $C\subset \Gm^r$ be a algebraic subgroup. Define:
$$G_C = (\GL{n_1}\times \dots\times \GL{n_r})/C. $$
We denote by $S_C$ the derived subgroup of $G_C$. There is a canonical isomorphism:
$$S_C \cong (\SL_{n_1}\times \dots\times \SL_{n_r})/C\cap (\mu_{n_1} \times\dots \times\mu_{n_r}).$$
Identify the group of characters $(\Gm^r)^*$ with $\bZ^r$. Let $M_C\subset \bZ^r$ be the submodule of all characters vanishing on $C$. By \cite[Theorem A.1]{Cernele2015} for any field extension $k\subset K$, $H^1(K,G_C)$ is in bijection with the following set of isomorphism classes of $r$-tuples of central simple algebras over $K$:
$$\Big\{(A_1,\dots,A_r) \mid \begin{matrix}
\text{For all } 1\leq j\leq r : & \deg(A_j) = n_j \\
\text{For all } (k_1,\dots,k_r) \in M_C:  & \sum{k_j [A_j]}=0 
\end{matrix}\Big\}$$
We will show how to determine if $S_C$ is strongly isotropic using $M_C$. We will need the existence of "sufficiently independent" central simple algebras. This is the content of the next lemma.

\begin{lemma}\label{existence of disjoint algebras}
Let $k$ be a field. For any natural numbers $r, n_1,\dots,n_r$ there exists a field extension $k\subset K$ and central division algebras $A_1,\dots,A_r$ over $K$ such that $\deg(A_i)=n_i$ and for any integers $k_1,\dots,k_r$:
$$\ind( \sum_j k_j [A_j] ) = \prod_j \frac{n_j}{\gcd(n_j,k_j)}.$$
\end{lemma}
\begin{proof}
Consider $G = \times^r_{j=1}\PGL_{n_j}$ as an algebraic group over $k$. Let $U = (U_1,\dots,U_r)$ be a versal torsor in $H^1(K,G) = \times^r_{j=1} H^1(K,\PGL_{n_j})$ for some field extension $k\subset K$. For all $1\leq j\leq r$, let $A_j$ be the central simple algebra of degree $n_j$ corresponding to $U_j$ as in \cite[29.10]{Knus1998}.
By \cite[Theorem 5.5]{Saltman1999} and \cite[Proposition 4.5.8]{Gille2006}, for any integers $k_1,\dots,k_r$ we have:
$$\exp( \sum_j k_j [A_j] ) \mid \ind( \sum_j k_j [A_j] ) \mid \prod_j \frac{\deg(A_j)}{\gcd(\deg(A_j),k_j)}  = \prod_j \frac{n_j}{\gcd(n_j,k_j)}.$$ Therefore it suffices to prove:
$$\exp( \sum_j k_j [A_j] ) = \prod_j \frac{n_j}{\gcd(n_j,k_j)}.$$
For all $1\leq j\leq r$ we define a Brauer invariant $f_j$ as follows. For any field extension $k\subset L$ and torsor $[\gamma]\in H^1(L,G)$ corresponding to an $r$-tuple of central simple algebras $(B_1,\dots,B_r)$ we define:
$$f_j(L)([\gamma]) = [B_j]\in \Br(L).$$
By \cite[Example 31.21]{Knus1998} (see also \cite[Theorem 2.4]{Blinstein2013}, \cite[Proposition 6.10]{Sansuc1981} for more details), we have an isomorphism between the group of characters of the fundamental group of $G$ and $\Inv(G,\Br)$. Using this isomorphism, one easily checks the following morphism is an isomorphism: $$F:\bigoplus_j \bZ/n_j\bZ \rightarrow \Inv(G,\Br),\ \ (k_1,\dots,k_r)\mapsto  \sum_j k_j f_j .$$ 
By \cref{Br invariants 2} the order of $F(k_1,\dots,k_r)(U)=  \sum_j k_j [A_j]$ in $\Br(K)$ is the same as the order of $(k_1,\dots,k_r)$ in $\bigoplus_j \bZ/n_j\bZ$ which is $ \prod_j \frac{n_j}{\gcd(n_j,k_j)}$ by elementary group theory. 
\end{proof}

We can now obtain a concrete criteria for strong isotropy of $S_C$ in terms of $C$.
The proof is inspired by the construction of ``generic" algebras in \cite[Definition 4.4]{Karpenko1999}.

\begin{theorem}\label{Strongly isotropic GLn's}
  The group $S_C$ is strongly isotropic if and only if for some $1\leq j\leq r$ there exists $(k_1,\dots,k_r)\in M_C$ such that:
$$(*):\ \ n_j \not \vert \ \frac{n_j}{\gcd(1+k_j, n_j)} \prod_{s\neq j} \frac{n_s}{\gcd(k_s, n_s)}. $$
\end{theorem}
\begin{proof}
By \cref{reduction to semisimple} it suffices to show $G_C$ is strongly isotropic if and only if for some $1\leq j\leq r$ there exists $(k_1,\dots,k_r)\in M_C$ such that $(*)$ is satisfied. For any $(k_1,\dots,k_r)\in M_C$ and tuple of algebras $(A_1,\dots,A_r)$ corresponding to a torsor in $H^1(K,G_C)$ we have by \cite[Theorem 5.5]{Saltman1999}:
\begin{align*}
    \ind([A_j]) &= \ind\big( [A_j] + \sum_s k_s[A_s]\big)  \mid \frac{n_j}{\gcd(1+k_j, n_j)} \prod_{s\neq j} \frac{n_s}{\gcd(k_s, n_s)}.
\end{align*}
Therefore if $(*)$ holds for some $(k_1,\dots,k_r)\in M_C$ then $A_j$ is not a division algebra and $G_C$ is strongly isotropic by \cref{criterion for isotropy An} and \cref{adjoint thm1}. We assume $(*)$ does not hold for any $(k_1,\dots,k_r)\in M_C$ and show $G_C$ is not strongly isotropic. By \cref{existence of disjoint algebras} after extending scalars we may choose a tuple $(A_1,\dots,A_r)$ of central division algebras such that for any integers $k_1,\dots,k_r$ we have:
$$\ind( \sum_j k_j [A_j] )  = \prod_j \frac{n_j}{\gcd(n_j,k_j)}.$$
Choose a basis $\{(k_{i 1},\dots,k_{i r})\}_{i=1,\dots,d}$ for $M_C$ and denote by $Y_i$ the Severi-Brauer variety of $A^{\otimes k_{i1}}_1\otimes\dots\otimes A^{\otimes k_{ir}}_r$ for all $1\leq i\leq d$. We put $K = k(Y_1\times\dots\times Y_d)$ and $B_j = (A_j)_K$ for all $j$. Since $k(Y)$ splits $A^{\otimes k_{i1}}_1\otimes\dots\otimes A^{\otimes k_{ir}}_r$ for all $i$, the tuple $(B_1,\dots,B_r)$ corresponds to a torsor in $H^1(K,G_C)$. By \cite[Theorem 2.3]{Schofield1992} we have:
\begin{align*}
    \ind([B_j]) &= \underset{a_1,\dots,a_d\in\bZ}{\gcd} \Big\{\ind\big( [A_j] + \sum_i a_i\sum_s k_{is}[A_s]\big)\Big\}\\
                &= \underset{(k_1,\dots,k_r)\in M_C}{\gcd} \Big\{\ind\big( [A_j] + \sum_s k_{s}[A_s]\big)\Big\}\\
                &=\underset{(k_1,\dots,k_r)\in M_C}{\gcd} \Big\{\frac{n_j}{\gcd(1+k_j, n_j)} \prod_{s\neq j} \frac{n_s}{\gcd(k_s, n_s)}.\Big\}.
\end{align*}
Since $(*)$ does not hold for any $(k_1,\dots,k_r)\in M_C$ it follows that $B_j$ is a division algebra for all $j$ and so the torsor corresponding to $(B_1,\dots,B_r)$ is anisotropic by  \cref{criterion for isotropy An} and \cref{adjoint thm1}.
\end{proof}

It is easy to use \cref{Strongly isotropic GLn's} to construct examples showing \cref{strongly isotropic quotient} fails without restrictions on the root system of $G$.

\begin{example}\label{example no quotient}
Let $\Delta: \Gm \to \Gm^2$ be the diagonal embedding and let $p$ be a prime . Define $G_C=\GL{p}\times \GL{p^2}/C$ where $C = \Delta(\Gm)$. We have $(k_1,k_2) = (1,-1)\in M_C$ and:
$$\frac{n_2}{\gcd(1+k_2, n_2)} \frac{n_1}{\gcd(k_1, n_1)}=\frac{p^2}{\gcd(0, p^2)} \frac{p}{\gcd(1, p)} = p. $$
Therefore $S_C\cong \SL_{p}\times \SL_{p^2}/\Delta(\mu_p)$ is strongly isotropic by \cref{Strongly isotropic GLn's}. Since $\SL_{p}\times \SL_{p^2}/\Delta(\mu_p)$ admits no strongly isotropic simple quotients this shows \cref{strongly isotropic quotient} fails without the restrictions on the root system.

\end{example}


\printbibliography

@Article{Berhuy2005,
  author  = {Berhuy, G. and Reichstein, Z.},
  journal = {Advances in Mathematics},
  title   = {{On the notion of canonical dimension for algebraic groups}},
  year    = {2005},
  issn    = {0001-8708},
  month   = dec,
  number  = {1},
  pages   = {128--171},
  volume  = {198},
  doi     = {10.1016/j.aim.2004.12.004},
  url     = {https://linkinghub.elsevier.com/retrieve/pii/S0001870805000447},
}

@Article{Merkurjev1996,
  author  = {Merkurjev, A. S. and Panin, I. A. and Wadsworth, A. R.},
  journal = {K-Theory},
  title   = {{Index reduction formulas for twisted flag varieties, I.}},
  year    = {1996},
  issn    = {1573-0514},
  month   = nov,
  number  = {6},
  pages   = {517--596},
  volume  = {10},
  doi     = {10.1007/BF00537543},
  url     = {http://www.portico.org/Portico/article?article=pgg197h1958},
}

@Book{Lam2004,
  author    = {Lam, T. Y.},
  publisher = {American Mathematical Society},
  title     = {{Introduction to Quadratic Forms over Fields}},
  year      = {2004},
  address   = {Providence, Rhode Island},
  isbn      = {9780821810958},
  month     = dec,
  series    = {Graduate Studies in Mathematics},
  volume    = {67},
  doi       = {10.1090/gsm/067},
  url       = {http://www.ams.org/gsm/067},
}

@Article{Sansuc1981,
  author  = {Sansuc, J.-J.},
  journal = {Journal f{\"{u}}r die reine und angewandte Mathematik (Crelles Journal)},
  title   = {{Groupe de Brauer et arithm{\'{e}}tique des groupes alg{\'{e}}briques lin{\'{e}}aires sur un corps de nombres.}},
  year    = {1981},
  issn    = {0075-4102},
  month   = sep,
  number  = {327},
  pages   = {12--80},
  volume  = {1981},
  doi     = {10.1515/crll.1981.327.12},
  url     = {https://www.degruyter.com/document/doi/10.1515/crll.1981.327.12/html},
}

@Article{Tits1990,
  author  = {Tits, J.},
  journal = {Journal of Algebra},
  title   = {{Strongly inner anisotropic forms of simple algebraic groups}},
  year    = {1990},
  issn    = {0021-8693},
  month   = jun,
  number  = {2},
  pages   = {648--677},
  volume  = {131},
  doi     = {10.1016/0021-8693(90)90201-X},
  url     = {https://linkinghub.elsevier.com/retrieve/pii/002186939090201X},
}

@Article{Blinstein2013,
  author  = {Blinstein, S. and Merkurjev, A.},
  journal = {Algebra {\&} Number Theory},
  title   = {{Cohomological invariants of algebraic tori}},
  year    = {2013},
  issn    = {1944-7833},
  month   = oct,
  number  = {7},
  pages   = {1643--1684},
  volume  = {7},
  doi     = {10.2140/ant.2013.7.1643},
  url     = {http://msp.org/ant/2013/7-7/p04.xhtml},
}

@incollection{Tits1966,
author = {Tits, J.},
booktitle = {Algebraic groups and discontinuous subgroups},
chapter = {1},
pages = {33--62},
publisher = {American Mathematical Society},
title = {{Classification of algebraic semi-simple groups}},
year = {1966}
}

@Article{Rosenlicht1956,
  author  = {Rosenlicht, M.},
  journal = {American Journal of Mathematics},
  title   = {{Some Basic Theorems on Algebraic Groups}},
  year    = {1956},
  issn    = {0002-9327},
  month   = apr,
  number  = {2},
  pages   = {401},
  volume  = {78},
  doi     = {10.2307/2372523},
  url     = {https://www.jstor.org/stable/2372523?origin=crossref},
}

@Article{Merkurjev1998,
  author  = {Merkurjev, A. S. and Panin, I. A. and Wadsworth, A. R.},
  journal = {K-Theory},
  title   = {{Index Reduction Formulas for Twisted Flag Varieties, II}},
  year    = {1998},
  issn    = {1573-0514},
  month   = jun,
  number  = {2},
  pages   = {101--196},
  volume  = {14},
  doi     = {10.1023/A:1007793218556},
  url     = {http://www.portico.org/Portico/article?article=pgg1zfp93hg},
}

@Article{Borel1965,
  author  = {Borel, A. and Tits, J.},
  journal = {Publications Math{\'{e}}matiques de l'IH{\'{E}}S},
  title   = {{Groupes r{\'{e}}ductifs}},
  year    = {1965},
  pages   = {55--151},
  volume  = {27},
}

@Book{GaribaldiS.MerkurjevA.Serre2003,
  author    = {{G}aribaldi, {S}. and {M}erkurjev, {A}. and {S}erre, J. P.},
  publisher = {American Mathematical Society},
  title     = {{Cohomological invariants in Galois cohomology}},
  year      = {2003},
}

@Article{Cernele2015,
  author  = {Cernele, S. and Reichstein, Z.},
  journal = {Pacific Journal of Mathematics},
  title   = {{Essential dimension and error-correcting codes}},
  year    = {2015},
  issn    = {0030-8730},
  month   = dec,
  number  = {1-2},
  pages   = {155--179},
  volume  = {279},
  doi     = {10.2140/pjm.2015.279.155},
  url     = {http://msp.org/pjm/2015/279-1/p08.xhtml},
}

@Book{Knus1998,
  author    = {Knus, M.-A. and Merkurjev, A. and Rost, M. and Tignol, J. P.},
  publisher = {American Mathematical Society},
  title     = {{The Book of Involutions}},
  year      = {1998},
  address   = {Providence, Rhode Island},
  isbn      = {9780821809044},
  month     = jun,
  series    = {Colloquium Publications},
  volume    = {44},
  doi       = {10.1090/coll/044},
  url       = {http://www.ams.org/coll/044},
}

@Article{Karpenko1999,
  author   = {Karpenko, N. A.},
  journal  = {Israel Journal of Mathematics},
  title    = {{Three theorems on common splitting fields of central simple algebras}},
  year     = {1999},
  issn     = {0021-2172},
  pages    = {125--141},
  volume   = {111},
  doi      = {10.1007/BF02810681},
}

@Book{Saltman1999,
  author    = {Saltman, D. J.},
  publisher = {American Mathematical Society; CBMS Regional Conference Series in Mathematics},
  title     = {{Lectures on Division Algebras}},
  year      = {1999},
}

@Article{Schofield1992,
  author  = {Schofield, A. and Bergh, M. V. D.},
  journal = {Transactions of the American Mathematical Society},
  title   = {{The Index of a Brauer Class on a Brauer-Severi Variety}},
  year    = {1992},
  issn    = {0002-9947},
  month   = oct,
  number  = {2},
  pages   = {729},
  volume  = {333},
  doi     = {10.2307/2154058},
  url     = {https://www.jstor.org/stable/2154058?origin=crossref},
}

@Article{Blanchet1991,
  author  = {Blanchet, A.},
  journal = {Communications in Algebra},
  title   = {{Function fields of generalized brauer-severi varieties}},
  year    = {1991},
  issn    = {0092-7872},
  month   = jan,
  number  = {1},
  pages   = {97--118},
  volume  = {19},
  doi     = {10.1080/00927879108824131},
  url     = {http://www.tandfonline.com/doi/abs/10.1080/00927879108824131},
}

@Book{Serre1997,
  author    = {Serre, J. P.},
  publisher = {Springer Berlin Heidelberg},
  title     = {{Galois Cohomology}},
  year      = {1997},
  address   = {Berlin, Heidelberg},
  isbn      = {978-3-540-42192-4},
  series    = {Springer Monographs in Mathematics},
  doi       = {10.1007/978-3-642-59141-9},
  url       = {http://link.springer.com/10.1007/978-3-642-59141-9},
}

@Article{Duncan2015,
  author        = {Duncan, A. and Reichstein, Z.},
  title         = {{Versality of algebraic group actions and rational points on twisted varieties}},
  year          = {2015},
  issn          = {1056-3911},
  number        = {3},
  pages         = {499--530},
  volume        = {24},
  archiveprefix = {arXiv},
  arxivid       = {1109.6093},
  booktitle     = {Journal of Algebraic Geometry},
  doi           = {10.1090/S1056-3911-2015-00644-0},
  eprint        = {1109.6093},
  isbn          = {2502172012},
}

@Book{Borel1991,
  author    = {Borel, A.},
  publisher = {Springer New York},
  title     = {{Linear Algebraic Groups}},
  year      = {1991},
  address   = {New York, NY},
  isbn      = {978-1-4612-6954-0},
  series    = {Graduate Texts in Mathematics},
  volume    = {126},
  doi       = {10.1007/978-1-4612-0941-6},
  url       = {http://link.springer.com/10.1007/978-1-4612-0941-6},
}

@Book{Milne2017,
  author    = {Milne, J. S.},
  publisher = {Cambridge University Press},
  title     = {{Algebraic Groups}},
  year      = {2017},
  address   = {Cambridge},
  isbn      = {9781316711736},
  number    = {5},
  volume    = {4},
  booktitle = {Cambridge press},
  doi       = {10.1017/9781316711736},
  issn      = {1573-8795},
  pages     = {463--482},
  url       = {http://ebooks.cambridge.org/ref/id/CBO9781316711736},
}

@Book{Springer1998,
  author    = {Springer, T. A.},
  publisher = {Birkh{\"{a}}user Boston},
  title     = {{Linear Algebraic Groups}},
  year      = {1998},
  address   = {Boston, MA},
  isbn      = {978-0-8176-4839-8},
  doi       = {10.1007/978-0-8176-4840-4},
  url       = {http://link.springer.com/10.1007/978-0-8176-4840-4},
}

@Book{Jacobson1996,
  author    = {Jacobson, N.},
  publisher = {Springer Berlin Heidelberg},
  title     = {{Finite-Dimensional Division Algebras over Fields}},
  year      = {1996},
  address   = {Berlin, Heidelberg},
  isbn      = {978-3-540-57029-5},
  doi       = {10.1007/978-3-642-02429-0},
  url       = {http://link.springer.com/10.1007/978-3-642-02429-0},
}

@Book{Gille2006,
  author    = {Gille, P. and Szamuely, T.},
  publisher = {Cambridge University Press},
  title     = {{Central Simple Algebras and Galois Cohomology}},
  year      = {2006},
  address   = {Cambridge},
  isbn      = {9780511607219},
  doi       = {10.1017/CBO9780511607219},
  url       = {http://ebooks.cambridge.org/ref/id/CBO9780511607219},
}

@Misc{Borovoi2010,
  author       = {M. Borovoi},
  howpublished = {MathOverflow},
  title        = {Splitting of a division algebra with an involution of second kind},
  year         = {2010},
  eprint       = {https://mathoverflow.net/q/24561},
  url          = {https://mathoverflow.net/q/24561},
}

\end{document}